\nonstopmode
\documentclass[10pt]{amsart}
\usepackage{latexsym}
\usepackage{fancyhdr}
\usepackage{amsmath, amssymb}
\usepackage[ansinew]{inputenc}
\usepackage[all]{xy}
\usepackage{pdflscape}
\usepackage{longtable}
\usepackage{rotating}
\usepackage{verbatim}
\usepackage{hyperref}
\usepackage{subfigure}
\usepackage{mathrsfs}
\usepackage{mdwlist}

\DeclareFontFamily{U}{mathb}{\hyphenchar\font45}
\DeclareFontShape{U}{mathb}{m}{n}{
      <5> <6> <7> <8> <9> <10> gen * mathb
      <10.95> mathb10 <12> <14.4> <17.28> <20.74> <24.88> mathb12
      }{}
\DeclareSymbolFont{mathb}{U}{mathb}{m}{n}
\DeclareFontSubstitution{U}{mathb}{m}{n}

\let\dot\relax
\DeclareMathAccent{\dot}{0}{mathb}{"39}
\let\ddot\relax
\DeclareMathAccent{\ddot}{0}{mathb}{"3A}
\let\dddot\relax
\DeclareMathAccent{\dddot}{0}{mathb}{"3B}
\let\ddddot\relax
\DeclareMathAccent{\ddddot}{0}{mathb}{"3C}

\theoremstyle{plain}
\newtheorem*{theorem*}{Theorem}
\newtheorem{theorem}{Theorem}[section]
\newtheorem*{lemma*}{Lemma}
\newtheorem{lemma}[theorem]{Lemma}
\newtheorem*{proposition*}{Proposition}
\newtheorem{proposition}[theorem]{Proposition}
\newtheorem*{corollary*}{Corollary}
\newtheorem{corollary}[theorem]{Corollary}
\newtheorem*{claim*}{Claim}

\newtheorem*{conjecture*}{Conjecture}

\newtheorem*{question*}{Question}
\theoremstyle{definition}
\newtheorem*{definition*}{Definition}
\newtheorem{definition}[theorem]{Definition}
\newtheorem*{example*}{Example}
\newtheorem{example}[theorem]{Example}

\newtheorem*{algorithm*}{Algorithm}
\newtheorem*{remark*}{Remark}
\newtheorem*{remarks*}{Remarks}
\newtheorem{remark}[theorem]{Remark}

\newtheorem*{convention*}{Convention}

\numberwithin{equation}{section}

\def\al{\alpha}
\def\be{\beta}
\def\ga{\gamma}
\def\de{\delta}
\def\ep{\epsilon}

\def\et{\eta}
\def\th{\theta}
\def\vt{\vartheta}

\def\la{\lambda}

\def\rh{\rho}

\def\si{\sigma}

\def\ta{\tau}

\def\ph{\phi}
\def\vh{\varphi}

\def\ps{\psi}
\def\om{\omega}

\def\Ga{\Gamma}

\def\Si{\Sigma}

\def\N{\mathbb{N}}

\def\R{\mathbb{R}}

\def\cB{\mathcal{B}}
\def\cC{\mathcal{C}}

\def\cJ{\mathcal{J}}

\def\fM{\mathfrak{M}}
\def\fN{\mathfrak{N}}

\def\fT{\mathfrak{T}}

\def\fW{\mathfrak{W}}

\def\p{\partial}

\def\<{\langle}
\def\>{\rangle}

\def\supp{\on{supp}}

\let\on=\operatorname

\newcommand{\sr}[1]%
{\ifmmode{}^\dagger\else${}^\dagger$\fi\ifvmode
\vbox to 0pt{\vss
 \hbox to 0pt{\hskip\hsize\hskip1em
 \vbox{\hsize3cm\raggedright\pretolerance10000
 \noindent #1\hfill}\hss}\vss}\else
 \vadjust{\vbox to0pt{\vss%
 \hbox to 0pt{\hskip\hsize\hskip1em%
 \vbox{\hsize3cm\raggedright\pretolerance10000%
 \noindent #1\hfill}\hss}\vss}}\fi%
}

\def\RR{\mathbb R}
\def\NN{\mathbb N}

\def\A{\;\forall}
\def\E{\;\exists}

\providecommand{\mapsfrom}{\kern.2em%
\setbox0=\hbox{$\leftarrow$\kern-.10em\rule[0.26mm]{0.1mm}{1.3mm}}\box0%
\kern.3em}

\title[Extension of Whitney jets of controlled growth]
{Extension of Whitney jets of controlled growth}

\author[A.~Rainer]{Armin Rainer}

\address{A.~Rainer: 
Fakult\"at f\"ur Mathematik, Universit\"at Wien, 
Oskar-Morgenstern-Platz~1, A-1090 Wien, Austria}
\email{armin.rainer@univie.ac.at}

\author[G.~Schindl]{Gerhard Schindl}

\address{G.~Schindl: Fakult\"at f\"ur Mathematik, Universit\"at Wien, 
Oskar-Morgenstern-Platz~1, A-1090 Wien, Austria}
\email{gerhard.schindl@univie.ac.at}

\begin{document}

\begin{abstract}
   We revisit Whitney's extension theorem in the ultradifferentiable Roumieu setting. 
   Based on the description of ultradifferentiable classes by weight matrices, we extend results 
   on how growth constraints on Whitney jets on arbitrary compact subsets in $\R^n$ are 
   preserved by their extensions to $\R^n$. 
   More precisely, for any admissible class $\cC$ of ultradifferentiable functions on $\R^n$
   we determine a class $\cC'$ such that 
   all ultradifferentiable Whitney jets of class $\cC'$ on arbitrary compact subsets admit 
   extensions in $\cC$. The class $\cC'$ can be explicitly computed from $\cC$.  
\end{abstract}

\thanks{Supported by FWF-Projects P~26735-N25 and J 3948-N35}
\date{\today}

\maketitle

\section{Introduction}

Whitney's classical extension theorem \cite{Whitney34a} provides conditions for the extension of jets defined in 
a closed subset of $\R^n$ to infinitely differentiable functions on $\R^n$. 
The present paper focuses on the question how growth constraints on the jets are preserved by their extension:
Let $E \subseteq \R^n$ be any compact subset and let $M = (M_k)$ be a positive sequence.  
A jet $F= (F^\al)_{\al} \in C^0(E,\R)^{\N^n}$ is said to be a \emph{Whitney jet of class $\cB^{\{M\}}$} if 
there exist $C,\rh>0$ such that 
\begin{gather}
  |F^\al(a)| \le C \rh^{|\al|} \,  M_{|\al|}, \quad \al \in \N^n,~ a \in E,
   \label{jets10}
  \\
  |(R^p_a F)^\al(b)| \le C \rh^{p+1} \, M_{p+1}\,  \frac{|b-a|^{p+1-|\al|}}{(p+1-|\al|)!}, \quad p \in \N,\, |\al| \le p,~ a,b \in E,  
  \label{jets20}
\end{gather}
where $(R^p_a F)^\al (x) := F^\al(x) - \sum_{|\be| \le p-|\al|} \be!^{-1} (x-a)^\be F^{\al+\be}(a)$.
Characterize the sequences $N=(N_k)$ with the property that 
every Whitney jet $F=(F^\al)_\al$ of class $\cB^{\{M\}}$ on $E$ admits an 
extension $f \in C^\infty(\R^n)$ such that there exist $\rh>0$ and $C\ge 1$ with  
\begin{equation}
  |f^{(\al)}(x)| \le C \rh^{|\al|} N_{|\al|},  \quad \al \in \N^n,~ x \in \R^n. 
\end{equation}
We denote the space of all such functions by $\cB^{\{N\}}(\R^n)$.

There is a vast literature on this problem and its variations. 
The problem as formulated was solved by Chaumat and Chollet \cite{ChaumatChollet94}: under the assumptions
\begin{enumerate}
  \item $M_k/k!$ is logarithmically convex,
  \item $M$ has \emph{moderate growth}, i.e., $M_{j+k}\le C^{j+k} M_j M_k$ for all $j,k$ and some constant $C$, 
  \item $N_k/k!$ is logarithmically convex,
  \item $N$ is \emph{non-quasianalytic}, i.e., $\sum_k N_{k-1}/N_k < \infty$,
\end{enumerate}
all Whitney jets of class $\cB^{\{M\}}$ on any compact $E \subseteq \R^n$ have extensions in $\cB^{\{N\}}(\R^n)$ if and only if 
\begin{equation} \label{condition}
  \sum_{\ell \ge k} \frac{N_{\ell-1}}{N_\ell} \lesssim \frac{k M_{k-1}}{M_k}.   
\end{equation} 
(For real valued functions $f$ and $g$ we write $f \lesssim g$ if and only if $f \le C g$ for some positive constant $C$.
We write $f \sim g$ if and only if  $f \lesssim g$ and $g \lesssim f$.)

Our main goal was to prove an analogous extension theorem without the rather restrictive assumptions of log-convexity and 
moderate growth. We were motivated by the fact that the related ultradifferentiable classes introduced 
by Beurling \cite{Beurling61} and Bj\"orck \cite{Bjoerck66} (see also Braun, Meise, and Taylor \cite{BMT90}) 
which are described by weight functions $\om$ can be equivalently represented by one parameter 
families $\fW = \{W^x\}_{x>0}$ of weight sequences (so-called \emph{weight matrices}) associated with $\om$; 
see \cite{RainerSchindl12}. 
The sequences $W^x$ typically do not have moderate 
growth and $W^x_k/k!$ is not log-convex.

We managed to completely dispense from the log-convexity condition and to replace the moderate growth assumption by 
some weaker conditions which are satisfied by weight matrices $\fW = \{W^x\}_{x>0}$ associated with suitable 
weight functions $\om$.
In fact, we replace the single sequence $N$ in the above problem 
by an \emph{admissible} weight matrix $\fN$ which incorporates these
weaker conditions and define the \emph{descendant} $M$ of $N \in \fN$ (see \ref{descendant}) 
which turns out to satisfy (1) and to be 
maximal with property \eqref{condition}. Our main result, Theorem \ref{extensiontheorem}, 
states that, for every descendant $M$ of $\fN$, 
all Whitney jets of class $\cB^{\{M\}}$ on any compact $E \subseteq \R^n$ have extensions in 
$\bigcup_{N \in \fN} \cB^{\{N\}}(\R^n)$. 
By a standard partition of unity argument, this can be generalized to arbitrary closed subsets of $\R^n$; therefore 
we restrict to compact sets.  
Combining our theorem with a result of Schmets and Valdivia \cite{SchmetsValdivia04}, 
we obtain a characterization of the extension property in the special case that the class is preserved by the extension:
all Whitney jets of class $\bigcup_{N \in \fN} \cB^{\{N\}}$ on any compact $E \subseteq \R^n$ have extensions in 
$\bigcup_{N \in \fN} \cB^{\{N\}}(\R^n)$ if and only if for each $M \in \fN$ there is $N \in \fN$ such that \eqref{condition}
holds; see Theorem \ref{characterization}.

Our main theorem generalizes the result of Chaumat and Chollet \cite{ChaumatChollet94}. Moreover, we 
reprove the extension theorem of Bonet, Braun, Meise, and Taylor \cite{BBMT91} for \emph{admissible} weight functions $\om$ 
by different methods. Beyond that, we deduce an extension result in the mixed weight function setting, see Corollary 
\ref{cor:mixed}, 
which to our knowledge  
was so far only considered in the special cases that either $E = \{0\}$ (by Bonet, Meise, and Taylor \cite{BonetMeiseTaylor92}) 
or that $E$ is convex with non-empty interior 
(by Langenbruch \cite{Langenbruch94}). 
Our method builds on the approach of Chaumat and Chollet \cite{ChaumatChollet94} who in turn 
combined the construction of optimal partitions of unity of Bruna \cite{Bruna80} with an extension procedure 
due to Dynkin \cite{Dynkin80}. 
This approach is quite direct and reproves the result for the special case $E = \{0\}$. 
By contrast, \cite{BBMT91} follows more closely Bruna's observation that the extension theorem for arbitrary compact sets $E$ 
is essentially a consequence of the result for $E = \{0\}$ and the existence of special cut-off functions   
which, in \cite{BBMT91}, are constructed using H\"ormanders $\overline \p$-method; the case $E = \{0\}$ 
for weight functions was treated by Bonet, Meise, and Taylor \cite{BonetMeiseTaylor89} and \cite{BonetMeiseTaylor92}. 
We do not know whether the approach of \cite{BBMT91} can be adapted to the mixed weight function setting.

In this paper we exclusively consider Roumieu type spaces.

\bigskip

\section{Weights}

\subsection{Weight sequences} \label{weights}

Let $\mu = (\mu_k)$ be a positive increasing sequence,  
$1 = \mu_0 \le \mu_1 \le \mu_2 \le \cdots$.
We associate the sequence $M=(M_k)$ given by 
\begin{equation} \label{def:M}
  M_k := \mu_0 \mu_1 \mu_2 \cdots \mu_k, 
\end{equation}
and the sequence $m = (m_k)$ defined by 
\begin{equation}
   M_k =: k!\, m_k.
\end{equation} 
We call $M$ a \emph{weight sequence} if $M_k^{1/k} \to \infty$. 

\begin{remark}
  We wish to warn the reader that some authors (e.g.\ \cite{ChaumatChollet94}, \cite{RainerSchindl12}) 
  prefer to work with \emph{``sequences without factorials''}, that is 
  $m_k$ instead of $M_k$. Consequently, conditions on weight sequences sometimes look slightly different 
  depending on the used convention. 
\end{remark}

Note that $\mu$ uniquely determines $M$ and $m$, and vice versa. In analogy we shall use 
$\nu \leftrightarrow N \leftrightarrow n$, $\si \leftrightarrow S \leftrightarrow s$, etc.
That $\mu$ is increasing means precisely that $M$ is logarithmically convex (log-convex for short). 
Log-convexity of $m$ is a stronger condition.

Since $\mu$ is increasing, \eqref{def:M} entails  
\begin{equation}\label{mucompare}
\A k\in \N_{>0} : M_k^{1/k} \le \mu_k,
\end{equation}
or equivalently, $M_k^{1/k}$ is increasing. Another consequence is $M_jM_k\le M_{j+k}$ for all $k,j$.

A weight sequence $M$ is called \emph{non-quasianalytic} if $\sum_k 1/\mu_k < \infty$; by the Denjoy--Carleman theorem 
(e.g.\ \cite[Theorem 1.3.8]{Hoermander83I})
this is the case if and only if the associated class of ultradifferentiable functions contains non-trivial elements 
with compact support.

Two weight sequences $M$ and $N$ are said to be \emph{equivalent} if 
$M_k^{1/k} \sim N_k^{1/k}$; a sufficient condition for this is $\mu \sim \nu$.
This means that the associated classes of ultradifferentiable functions coincide; see Section \ref{DC}.

\subsection{Associated functions} \label{hGaSi}

The following facts are well-known; we refer to \cite{Mandelbrojt52} and \cite{ChaumatChollet94}.
With a weight sequence $M$ we associate the function  
\begin{equation} \label{h}
  h_M(t) := \inf_{k \in \N} M_k t^k, \quad t > 0, \quad h_M(0):=0.
\end{equation}
Then $h_M$ is increasing, continuous, and positive for $t>0$. For $t \ge 1/\mu_1$ we have $h_M(t) = 1$. From $h_M$ 
we may recover the sequence $M$ by $M_k = \sup_{t>0} t^{-k} h_M(t)$.
We associate the counting function $\Ga_M$ by setting 
\begin{equation} \label{counting2}
  \Ga_M(t) := \min\{k : h_M(t) =  M_k t^k\} = \min\{k : \mu_{k+1} \ge t^{-1} \};
\end{equation}
for this identity we need that $\mu$ is increasing. Then:
\begin{gather}
	\text{$k \mapsto M_k t^k$ is decreasing for $k \le \Ga_M(t)$,} \label{GaProp1}
	\\
	\text{$h_M(t) = M_{\Ga_M(t)} t^{\Ga_M(t)} \le M_k t^k$ for all $k$.} \label{GaProp2} 
\end{gather}
We shall also use  
\begin{align} \label{counting}
  \Si_M(t) := |\{k \ge 1 : \mu_k \le t\}| = \max \{k : \mu_k \le t\}.
\end{align}
Note that 
\begin{equation*} 
  \Ga_M(t^{-1}) \le \Si_M(t) \text{ for all } t>0, \text{ and } \Ga_M(t^{-1}) = \Si_M(t) \text{ if } t \not \in \{\mu_k\}_k.
\end{equation*}
It is well-known (cf.\ \cite{Mandelbrojt52} and \cite{Komatsu73}) that $\om_M$ defined by $h_M(t) = \exp(-\om_M(1/t))$ satisfies
\begin{equation} \label{omM}
  \om_M(t) = \int_{0}^t\frac{\Si_M(u)}{u} \,du.
\end{equation}

\subsection{Moderate growth}

A weight sequence $M$ is said to have \emph{moderate growth} if the equivalent conditions of 
the following lemma are satisfied; some authors call it  
\emph{stability under ultradifferentiable operators}, e.g.\ \cite[(M.2)]{Komatsu73}.

\begin{lemma}\label{ChaumatCholletmod}
Let $M$ be a weight sequence. The following are equivalent:
\begin{enumerate}
\item[(0)] $\exists C\ge 1 \A j,k\in\N : m_{j+k}\le C^{j+k} m_j m_k$.
\item $\exists C\ge 1 \A j,k\in\N : M_{j+k}\le C^{j+k} M_j M_k$.
\item $\mu_k \lesssim M_k^{1/k}$. 
\item $\mu_{2k} \lesssim \mu_k$. 
\item $\exists C\ge 1 \A t >0 :  2\Si_M(t) \le \Si_M(C t)$.
\item $\exists C\ge 1 \A t\ge 0 : 2\om_M(t) \le \om_M(Ct) + C$.
\end{enumerate}
\end{lemma}

\begin{proof}
  Most of this is well-known. $(0) \Leftrightarrow (1)$ since $1 \le \binom{k+j}{j} \le 2^{k+j}$.
  For $(1) \Leftrightarrow (3)$ see \cite[Appendix B]{Matsumoto87} and for
  $(1) \Leftrightarrow (5)$ see \cite[Proposition 3.6]{Komatsu73}.

  $(1) \Rightarrow (2)$ We have $\mu_{k}^k \le  \mu_{k+1}^k\le\mu_{k+1}\cdots\mu_{2k}= M_{2k}/M_k\le C^{2k}M_k$.

  $(2) \Rightarrow (1)$ Note that $\mu_k \le C\, M_k^{1/k}$ if and only if $M_k^{1/k}\le C^{1/(k-1)}\,M_{k-1}^{1/(k-1)}$. 
  By iteration,  
  $M_{2k}^{1/(2k)}\le  C\, M_k^{1/k}$ and thus $M_{2k}\le C^{2k} M_k^2$. By \cite[Theorem 1]{Matsumoto84}, this implies $(1)$.

  $(3) \Rightarrow (4)$ follows from the definition of $\Si_M$.

  $(4)  \Rightarrow (5)$ By \eqref{omM},
  \begin{equation*}
  2\omega_M(t) = \int_{0}^t\frac{2 \Si_M(s)}{s}ds\le\int_{0}^t\frac{\Si_M(Cs)}{s}ds = \int_{0}^{Ct}\frac{\Si_M(\sigma)}{\sigma}d\sigma=\omega_M(Ct).
  \qedhere
  \end{equation*} 
\end{proof}

\begin{remark}
  In \cite{ChaumatChollet94} and \cite{ChaumatChollet98} a weight sequence is said to have moderate growth 
  if it satisfies $(1)$ and 
  $M_{k+1}^k \le C^k \, M_k^{k+1}$ for all $k$. It is easy to see that the latter condition is equivalent to $(2)$; 
  so it is superfluous by the 
  lemma.  
\end{remark}

The proof of the lemma shows that $M$ having moderate growth  
is also equivalent to $\mu_{k+1} \lesssim M_k^{1/k}$, and with \eqref{mucompare} we obtain,
\begin{equation} \label{wmg}
   \mu_{k+1} \lesssim \mu_k.  
\end{equation} 
This condition means that the sequence $\log M_k$ is \emph{almost concave} in the sense that there is a constant $C>0$ such that 
\begin{equation*}
  \log M_{k-1} + \log M_{k+1} \le 2 \log M_k + C
\end{equation*}
for all $k$.   
There are weight sequences of non-moderate growth that satisfy \eqref{wmg}, 
for instance, $M_k= A^{k^p}$, where $A>1$ and $0< p \le 2$; cf.\ Section \ref{example}.

The weight sequences associated with a weight function $\om$ (see Section \ref{associatedmatrix}) 
do in general not have moderate growth
(nor is the weaker condition \eqref{wmg} guaranteed); cf.\ Section \ref{example}.
However, we shall see in Lemma \ref{lemma4} below that all associated sequences \emph{together} fulfill a 
\emph{moderate growth condition}.  
In the next lemma we 
show how this condition translates to the functions $h_M$ and $\Ga_M$. 
  
\begin{lemma} \label{claim1}
  Let $M=(M_k)$ and $\dot M=(\dot M_k)$ be weight sequences such that
  \begin{equation} \label{eq31}
    \mu_{2k} \lesssim \dot \mu_k.   
  \end{equation}
  Then
  \begin{gather}
      \E C \ge 1 \A t >0 : h_M(t) \le h_{\dot M}(Ct)^2, \label{hmg}
      \\
      \E \la < 1 \A t>0 : 2  \Ga_{\dot M}(t) \le \Ga_M(\la t).   \label{eq32}
  \end{gather}
\end{lemma}

\begin{proof}
  Condition \eqref{eq31} implies $2 \Si_{\dot M}(t) \le \Si_M (Ct)$ for all $t>0$ and some $C \ge 1$, by \eqref{counting}.
  Using \eqref{omM}, we obtain $2\omega_{\dot M}(t) \le \omega_M(Ct)$ for all $t>0$, which is clearly equivalent to 
  \eqref{hmg}.
  Similarly, \eqref{eq31} implies $2 \Ga_{\dot M}(Ct) \le \Ga_M (t)$ for all $t>0$, i.e., 
  \eqref{eq32}.
\end{proof}

\begin{remark} \label{rem:claim1}
  Note that \eqref{hmg} is equivalent to 
  \begin{equation}
    \E C \ge 1 \A k,j \in \N : M_{k+j} \le C^{k+j} \dot M_j \dot M_{k}. \label{mg0}
  \end{equation}
  In fact, that \eqref{hmg} implies \eqref{mg0} was shown in \cite[Proposition 3.6]{Schindl15}.
  For the opposite direction, note that $N_k := \min_{0 \le j \le k} \dot  M_j  \dot M_{k-j}$ is log-convex, 
  by \cite[Lemma 3.5]{Komatsu73}. By \eqref{mg0}, $M_k \le C^k N_k$ for all $k$, and thus
  $h_{M}(t) =   \inf_{k}  M_k t^k \le \inf_{k}  N_k (Ct)^k =  h_{N}(C t)$.
  Now \eqref{hmg} follows from the fact that
  $h_{N} = (h_{\dot M})^2$, see \cite[Lemma 3.5]{Komatsu73}.
\end{remark}

\subsection{Weight functions} \label{weightfunction}

A \emph{weight function} is a continuous increasing function $\om : [0,\infty) \to [0,\infty)$ with $\om|_{[0,1]} =0$ and 
$\lim_{t \to \infty} \om(t) = \infty$ that satisfies
\begin{align}
   & \om(2t) = O(\om(t)) \quad\text{ as } t \to \infty, \label{om1}\\
   & \om(t) = O(t) \quad\text{ as } t \to \infty, \label{om2}\\
   & \log t = o(\om(t)) \quad\text{ as } t \to \infty, \label{om3}\\
   & \vh(t) := \om(e^t) \text{ is convex}.  \label{om4}
\end{align}
For a weight function $\om$ we consider the \emph{Young conjugate} $\vh^*$ of $\vh$,
\[
  \vh^*(x) := \sup_{y\ge 0} xy-\vh(y), \quad x \ge 0,
\] 
which is a convex increasing function satisfying $\vh^*(0)=0$, $\vh^{**}=\vh$, and $x/\vh^*(x) \to 0$ as $x \to \infty$; 
cf.\ \cite{BMT90}.

\subsection{The weight matrix associated with a weight function} \label{associatedmatrix}

With a weight function $\om$ we associate a \emph{weight matrix} $\fW = \{W^x\}_{x>0}$ by setting 
\begin{equation*}
  W^x_k := \exp(\tfrac{1}{x}\vh^*(x k)), \quad k \in \N;
\end{equation*}
cf.\ \cite[5.5]{RainerSchindl12}.
By the properties of $\vh^*$, each $W^x$ is a weight sequence (in the sense of Section \ref{weights}). 
Moreover, setting $\vt^x_k := W^x_k/W^x_{k-1}$ we have $\vt^x \le \vt^y$ 
if $x \le y$, which entails $W^x \le W^y$.

\begin{lemma} \label{lemma4} 
  For all $x>0$ and all $k \in \N_{\ge 2}$,
  $\vt^x_{2k} \le   \vt^{4x}_{k}$.
\end{lemma}

\begin{proof}
  The inequality $\vt^x_{2k} \le   \vt^{4x}_{k}$ is equivalent to
  \begin{equation*}
    \frac1x (\vh^*(2k x) - \vh^*((2k-1)x) ) \le \frac1{4x} (\vh^*(4k x) - \vh^*(4(k-1)x) ),    
  \end{equation*} 
  which follows from the convexity of $\vh^*$. 
\end{proof}

\section{Spaces of functions and jets}

\subsection{Ultradifferentiable functions} \label{DC}

Let $M=(M_k)$ be a weight sequence and $\rh >0$.
We consider the Banach space $\cB^{M}_\rh(\R^n) := \{f \in C^\infty (\R^n) : \|f\|^M_{\rh}< \infty\}$, where 
\[
  \|f\|^M_{\rh} := \sup_{x \in \R^n,\, \al \in \N^n} \frac{|\p^\al f(x)|}{\rh^{|\al|} M_{|\al|}},
\]
and the inductive limit 
\begin{equation} \label{DCclass}
  \cB^{\{M\}}(\R^n) := \on{ind}_{\rh \in \N} \cB^{M}_\rh(\R^n).  
\end{equation}
For weight sequences $M$ and $N$ we have $\cB^{\{M\}} \subseteq \cB^{\{N\}}$ if and only if 
$M_k^{1/k} \lesssim N_k^{1/k}$; one implication is obvious, the other follows from the existence of 
\emph{characteristic} $\cB^{\{M\}}$-functions, cf. \cite[Lemma 2.9 and Proposition 2.12]{RainerSchindl12}.

Let $\om$ be a weight function and $\rh >0$.  
We consider the Banach space $\cB^{\om}_\rh(\R^n) := \{f \in C^\infty (\R^n) : \|f\|^\om_{\rh}< \infty\}$, 
where  
\[
  \|f\|^\om_{\rh} := \sup_{x \in \R^n,\,\al \in \N^n} |\p^\al f(x)| \exp(-\tfrac{1}{\rh} \vh^*(\rh |\al|)),
\]
and the inductive limit 
\begin{equation}
  \cB^{\{\om\}}(\R^n) := \on{ind}_{\rh \in \N} \cB^{\om}_\rh(\R^n).
\end{equation}
For weight functions $\om$ and $\si$ we have $\cB^{\{\om\}} \subseteq \cB^{\{\si\}}$ if and only if
$\si(t) = O(\om(t))$ as $t \to \infty$;
cf.\ \cite[Corollary 5.17]{RainerSchindl12}.

The associated weight matrix $\fW$ allows us to describe any class $\cB^{\{\om\}}(\R^n)$ as 
a union of spaces of type \eqref{DCclass}:

\begin{theorem}[{\cite[Corollaries 5.8 and 5.15]{RainerSchindl12}}] \label{representation}
  Let $\om$ be a weight function and let $\fW = \{W^x\}_{x>0}$ be the associated weight matrix. 
  Then, as locally convex spaces,
  \begin{align*}
    \cB^{\{\om\}}(\R^n) &= \on{ind}_{x > 0} \cB^{\{W^x\}}(\R^n) = \on{ind}_{x > 0} \on{ind}_{\rh>0} \cB^{W^x}_\rh(\R^n).  
  \end{align*}
  We have $\cB^{\{\om\}}(\R^n) = \cB^{\{W^x\}}(\R^n)$ for all $x>0$ if and only if 
  \begin{align}
     \E H\ge 1 \A t\ge 0 : 
      2\om(t) \le \om(Ht) + H.   \label{om6}
  \end{align} 
  Moreover, \eqref{om6} holds if and only if some (equivalently each) $W^x$ has moderate growth. 
\end{theorem}

\begin{remark}
  Let us emphasize that the fact that $\cB^{\{\om\}} = \cB^{\{M\}}$ for some weight sequence $M$ 
  if and only if $\om$ satisfies \eqref{om6} is due to \cite{BMM07}.
\end{remark}

Motivated by this result we define a \emph{weight matrix} to be a  
family $\fM$ of weight sequences which is totally ordered with respect to the pointwise order relation on sequences, i.e., 
\begin{enumerate}
   \item $\fM \subseteq \R^\N$,
   \item each $M \in \fM$ is a weight sequence in the sense of Section \ref{weights},
   \item for all $M,\dot M \in \fM$ we have $M \le \dot M$ or $M \ge \dot M$. 
 \end{enumerate} 
For a weight matrix $\fM$  
we consider 
\begin{align*}
    \cB^{\{\fM\}}(\R^n) &:= \on{ind}_{M \in \fM} \cB^{\{M\}}(\R^n) = \on{ind}_{M \in \fM} \on{ind}_{\rh>0} \cB^{M}_\rh(\R^n). 
\end{align*}
For weight matrices $\fM$ and $\fN$ we have $\cB^{\{\fM\}} \subseteq \cB^{\{\fN\}}$ if and only if
$\A M \in \fM \E N \in \fN : M_k^{1/k} \lesssim N_k^{1/k}$; 
cf.\ \cite[Proposition 4.6]{RainerSchindl12}.

\subsection{Whitney jets of controlled growth}

Let $E$ be a compact subset of $\R^n$. 
We denote by $\cJ^\infty(E)$ the vector space of all jets $F= (F^\al)_{\al\in \N^n} \in C^0(E,\R)^{\N^n}$ on $E$. 
For $a \in E$ and $p \in \N$ we associate the Taylor polynomial
\begin{gather*}
  T^p_a : \cJ^\infty(E) \to C^\infty (\R^n,\R), ~ F \mapsto T^p_a F(x) := \sum_{|\al|\le p} \frac{(x-a)^\al}{\al!} F^\al(a),    
\end{gather*}
and the remainder $R^p_a F = ((R^p_a F)^\al)_{|\al| \le p}$ with
\begin{gather*}
  (R^p_a F)^\al (x) := F^\al(x) - \sum_{|\be| \le p-|\al|} \frac{(x-a)^\be}{\be!} F^{\al+\be}(a), \quad a,x \in E.      
\end{gather*}
Let us denote by 
$j^\infty_E$ the mapping which assigns to a $C^\infty$-function $f$ on $\R^n$
the jet $j^\infty_E(f) := (\p^\al f|_E)_\al$. 
By Taylor's formula, $F  =j^\infty_E(f)$ satisfies
\begin{equation*}\label{Whitneyjets}
  (R^p_a F)^\al (x) = o(|x-a|^{p-|\al|}) \quad \text{ for $a,x \in E$, $p\in \N$, $|\al| \le p$ as $|x-y|\to 0$.}
\end{equation*}
Conversely, if a jet $F \in \cJ^\infty(E)$ has this property, then it admits a 
$C^\infty$-extension to $\R^n$,
by Whitney's extension theorem \cite{Whitney34a} (for modern accounts see e.g.\ 
\cite[Ch.~1]{Malgrange67}, \cite[IV.3]{Tougeron72}, or 
\cite[Theorem 2.3.6]{Hoermander83I}).

\begin{definition}[Ultradifferentiable Whitney jets]
Let $E \subseteq \R^n$ be compact.
Let $M=(M_k)$ be a weight sequence.
For fixed $\rh>0$ we denote by $\cB^{M}_\rh(E)$ the set of all jets $F$ satisfying \eqref{jets10} and \eqref{jets20}; 
the smallest constant $C$ defines a complete norm on $\cB^{M}_\rh(E)$. We define 
$$\cB^{\{M\}}(E) := \on{ind}_{\rh \in \N} \cB^{M}_\rh(E).$$
An element of $\cB^{\{M\}}(E)$ is called a \emph{Whitney jet of class $\cB^{\{M\}}$ on $E$}.

Let $\fM$ be a weight matrix. 
A jet $F$ is said to be a \emph{Whitney jet of class $\cB^{\{\fM\}}$ on $E$} if $F \in \cB^{\{M\}}(E)$ 
for some $M \in \fM$; we set 
$$\cB^{\{\fM\}}(E)= \on{ind}_{M \in \fM} \cB^{\{M\}}(E) = \on{ind}_{M \in \fM} \on{ind}_{\rh>0} \cB^{M}_\rh(E).$$

Let $\om$ be a weight function and $\fW$ the associated weight matrix.
A jet $F$ is said to be a \emph{Whitney jet of class $\cB^{\{\om\}}$ on $E$} if $F \in \cB^{\{\fW\}}(E)$; 
the topology on $\cB^{\{\om\}}(E)$ is given by the identification 
$\cB^{\{\om\}}(E) = \cB^{\{\fW\}}(E)$. 
\end{definition}

\begin{remark}
  This definition of Whitney jet of class $\cB^{\{\om\}}$ on $E$ coincides with the one given in 
  \cite{BBMT91}. This follows from the fact that
  any weight matrix $\fW$ associated with a weight function has the following property: 
  \begin{equation} \label{5.10}
    \A \rh>0 \E H\ge 1 \A x >0 \E C \ge 1 \A k \in \N : \rh^k W^x_k \le C W^{Hx}_k;
  \end{equation}
  cf.\ \cite[Lemma 5.9]{RainerSchindl12}.
\end{remark}

\section{A special partition of unity}

\subsection{The descendant of a non-quasianalytic weight sequence} \label{descendant}

Following an idea in the proof of \cite[Proposition 1.1]{Petzsche88} we associate with any non-quasianalytic 
weight sequence $N$ a weight sequence $S$ with many good properties.

\begin{definition}[Descendant]
  Let $\nu=(\nu_k)$ be an increasing positive sequence with $\nu_0 =1$ and  $\sum_k 1/\nu_k < \infty$. Let us associate a 
positive sequence $\si = \si(\nu)$ in the following way. 
We define 
  \begin{equation} \label{tau}
    \ta_k := \frac{k}{\nu_k} + \sum_{j\ge k} \frac 1 {\nu_j}, \quad k \ge 1,
  \end{equation}
and set 
  \begin{equation} \label{sigma}
    \si_k := \frac{\ta_1 k}{\ta_k}, \quad k \ge 1, \quad \si_0 := 1.  
  \end{equation}
We say that $\si$ is \emph{the descendant of $\nu$}; we shall also say that 
$S_k=\si_0 \si_1 \cdots \si_k$ is the descendant of $N_k = \nu_0 \nu_1 \cdots \nu_k$.  
\end{definition}

We shall also use the abbreviations $\si^*_k := \si_k/k$ and $\nu^*_k := \nu_k/k$.

\begin{lemma} \label{lem:log-convex}
  Let $\si$ be the descendant of $\nu$. Then:
  \begin{enumerate}
    \item $\si \lesssim \nu$.
    \item $\sum_{j\ge k}  1 /\nu_j \lesssim k/\si_k$.
    \item $1 \le  \si^*_k$ is increasing to $\infty$. 
    \item $\si_{k+1} \lesssim \si_k$ if and only if 
    \begin{equation} \label{ugly}
     \E C>0 \A k \in \N : \frac{\nu_{k+1}}{\nu_k} \le (C+1) + C \nu_{k+1}^* \sum_{j \ge k+1} \frac 1{\nu_j},
    \end{equation}
    in particular, if $\nu_{k+1} \lesssim \nu_k$.
    \item If $\mu$ is an increasing positive sequence satisfying $\mu  \lesssim \nu$ 
    and $\sum_{j\ge k}  1 /\nu_j \lesssim k/\mu_k$, then $\mu \lesssim \si$, i.e., $\si$ is the 
    \emph{largest} sequence satisfying \thetag{1} and \thetag{2}. 
    \item Let $\dot \nu$ be another increasing positive sequence with $\sum_k 1/\dot \nu_k < \infty$ and $\dot \si$ 
    its descendant. Then $\nu_{2k} \lesssim \dot \nu_{k}$ implies $\si_{2k} \lesssim \dot \si_{k}$. 
  \end{enumerate} 
\end{lemma}

\begin{proof}
  (1), (2), and (3) are immediate. 

  (4) This follows by a straightforward computation using that $\si_{k+1} \lesssim \si_k$ is equivalent to $\ta_k \lesssim 
  \ta_{k+1}$ and
  \begin{align*}
   \ta_{k} - \ta_{k+1} =(k+1) \Big(\frac{1}{\nu_k} - \frac{1}{\nu_{k+1}}\Big).
  \end{align*}

  (5) The assumptions on $\mu$ imply 
  \begin{align*}
    \ta_k \lesssim \frac{k}{\nu_k} +  \frac{k}{\mu_k} \lesssim \frac{k}{\mu_k}. 
  \end{align*}

  (6) If $\nu_{2k} \lesssim \dot \nu_k$, then 
  \begin{align*}
    \dot \ta_k = \frac{k}{\dot \nu_k} + \sum_{j\ge k} \frac 1 {\dot \nu_j} 
    \lesssim \frac{k}{\nu_{2k}} + \sum_{j\ge k} \frac 1 {\nu_{2j}} 
    \le \frac{2k}{\nu_{2k}} + \sum_{j\ge 2k} \frac 1 {\nu_{j}} = \ta_{2k}
  \end{align*}
 which implies $\si_{2k} \lesssim \dot \si_{k}$. 
\end{proof}

\begin{remark}
  Notice that the descendant $S$ of $N$ has the property that even $s_k = S_k/k!$ is a weight sequence 
  (not only $S$). Hence we can work with the functions $h_s$, $\Ga_s$, and $\Si_s$ introduced in Section \ref{hGaSi}. 

  The descendent $\si$ is equivalent to $\nu$, i.e., $\si \sim \nu$, if and only if
  $\sum_{j\ge k}  1 /\nu_j \lesssim k/\nu_k$; this follows from (1), (2), and (5) in Lemma \ref{lem:log-convex}.
  The latter is the so-called \emph{strong non-quasianalyticity} condition (cf.\ \cite{Komatsu73}, \cite{Bruna80}, 
  \cite{ChaumatChollet94}, and \cite{Petzsche88}, where it is condition ($\ga_1$)). 
  It is well-known that, if $\nu^*$ is increasing and $N$ has moderate growth, then the strong non-quasianalyticity 
  condition is equivalent to $j_E^\infty : \cB^{\{N\}}(\R^n) \to \cB^{\{N\}}(E)$ being surjective for every 
  compact $E \subseteq \R^n$; see e.g.\ \cite[Theorem 30]{ChaumatChollet94} and Section \ref{ssec:characterization} below.

  Moreover, we want to remark that one can recover a \emph{predecessor} $\nu$ from its descendant $\si$. 
  Provided that a positive sequence $\si$ satisfying \ref{lem:log-convex}(3) is given one can choose $\ta_1 := 2$ 
  and $\nu_1 :=1$ 
  and then solve the equations \eqref{tau} and \eqref{sigma} recursively for $\nu_k$. 
  The resulting sequence $\nu$ is increasing, satisfies $\sum_{k\ge 1} 1/\nu_k =1$, and its descendant is $\si$.	
\end{remark}

\subsection{A special partition of unity}

We construct a partition of unity which will be crucial for the proof of the extension theorem.
The idea goes back to Bruna \cite{Bruna80} who considered a single weight sequence $M$ satisfying 
$\sum_{\ell \ge k} 1/\mu_\ell \lesssim k/\mu_k$. 
Chaumat and Chollet \cite{ChaumatChollet94} extended the construction to the case of two weight sequences $M$ and $N$
satisfying $\sum_{\ell \ge k} 1/\nu_\ell \lesssim k/\mu_k$.
We make adjustments to this construction in order to compensate for the moderate growth condition which was heavily used 
in \cite{ChaumatChollet94} and \cite{Bruna80}.  

Recall that $\si^*_k = \si_k/k$ and $s_k = \si^*_1 \cdots \si^*_k$.

\begin{lemma}\label{sequenceconstruction}
Let $\nu$ be an increasing positive sequence satisfying $\nu_0 =1$, $\sum_k 1/\nu_k < \infty$, and \eqref{ugly}.   
Let $\dot \nu$ be another increasing positive sequence such that $\nu_k \lesssim \dot N_k^{1/k}$.
Let $\si$ be the descendant of $\nu$.
There is a constant $A\ge 1$ such that for all integers $p\ge 1$ there exists a sequence $(\alpha^p_k)_{k\in\NN}$ satisfying
\begin{align}
&\sum_{k\ge 0}\frac{\alpha^p_k}{\alpha^p_{k+1}}\le 1, \quad \alpha_0^p=1,\\
&0<\alpha^p_k\le \Big(h_{s}\Big(\frac{1}{3\si^*_{p}}\Big)\Big)^{-1}\Big(\frac{A}{\si^*_{p+1}}\Big)^k \dot N_k. 
\label{alpha3}
\end{align}
\end{lemma}

\begin{proof}
  Let $A \ge 1$ be a constant which shall be specified later. Define 
  \[
    \al^p_k := 
    \begin{cases}
        (A/\si^*_{p+1})^k \dot N_k & \text{ if } k >p, \\
        (2p)^k & \text{ if } k \le p.
    \end{cases}
  \]
  By Lemma \ref{lem:log-convex},  $\si_{k+1} \lesssim \si_k \lesssim \nu_k \lesssim \dot N_k^{1/k}$ and thus, 
  for some constant $C \ge 1$, 
  \begin{align*}
    \alpha^p_p= 2^pp^p\le  \Big(\frac{2C}{\si^*_{p+1}}\Big)^p \dot N_p. 
  \end{align*}
  So, for $k \ge p$,
  \[
    \frac{\alpha_k^p}{\alpha^p_{k+1}} \le \frac{\si^*_{p+1}}{ A\,  \dot\nu_{k+1}},
  \]
  provided that $A \ge 2C$.
  Hence, since $\sum_{j\ge k}  1 /\dot \nu_j \lesssim \sum_{j\ge k}  1 /\nu_j \lesssim k/\si_k$ 
  by \eqref{mucompare} and Lemma~\ref{lem:log-convex},
  \begin{align*}
      \sum_{k\ge 0} \frac{\alpha_k^p}{\alpha^p_{k+1}}
      &\le\sum_{k<p}\frac{1}{2p}+ \frac{\si^*_{p+1}}{A}\sum_{k\ge p}\frac{1}{\dot \nu_{k+1}}
      \le 1,
  \end{align*}
  if $A$ is chosen large enough. Since always $h_s\le 1$, \eqref{alpha3} is obvious for $k>p$. 
  If $k\le p$, then, using $\si_{k+1} \lesssim \si_k$ and $\si \lesssim \dot \nu$,
  \begin{align*}
  \frac{\al^p_k}{(A/\si^*_{p+1})^k\,  \dot N_k}
  &=\frac{2^k p^k}{(A/\si^*_{p+1})^k\,  \dot N_k}
  \le \frac{\si_{p+1}^k}{(A/2)^{k} \dot N_k} \le \frac{\si_{p}^k}{S_k}, 
  \end{align*}
  provided that $A$ is large enough. 
  Since $\si_{p}^k/S_k \le \si_{p}^p/S_p$ if $k\le p$, we obtain
  \begin{align*}
  \frac{\al^p_k}{(A/\si^*_{p+1})^k\,  \dot N_k}
  &\le \frac{\si_{p}^p}{S_p} = \frac{\si_{p}^p}{p!\,s_p}
  \le \frac{(3 \si^*_{p})^p}{s_p} \le \Big(h_{s}\Big(\frac{1}{3\si^*_{p}}\Big)\Big)^{-1}, 
  \end{align*} 
  by \eqref{h}.
  The proof is complete.
\end{proof}

\begin{proposition}\label{testfunctionconstruction}
In the setup of Lemma \ref{sequenceconstruction}
there is a constant $B\ge 1$ such that for all $\ep>0$ and all $t>1$ there exists a $C^\infty$-function $\ph_{\ep,t}$ satisfying
\begin{enumerate}
\item $0\le\ph_{\ep,t}(x)\le 1$ for all $x\in\RR^n$,
\item $\ph_{\ep,t}(x)=1$ for all $x\in\RR^n$ with $|x|\le 1$,
\item $\ph_{\ep,t}(x))=0$ for all $x\in\RR^n$ with $|x|\ge t$,
\item for all $\be \in\NN^n$ and all $x\in\RR^n$,
\begin{equation*} 
  |\ph^{(\be)}_{\ep,t}(x)|\le \frac{\ep^{|\be|} \dot N_{|\be|}}{h_{s}(B\ep(t-1))}.
\end{equation*}
\end{enumerate}
\end{proposition}

\begin{proof}
It suffices to consider the case $n=1$ and $t=2$; the general case follows by composition with suitable functions, e.g., 
$\ph_{\ep,t}(x) := \ph_{(t-1)\ep,2}(\th(x))$ where $\th$ is an odd diffeomorphism of $\R$ satisfying 
$\th(x) = (x+t-2)/(t-1)$ for $x\ge 1$.

Let $A$ be the constant from Lemma \ref{sequenceconstruction}.
Fix $0<\eta\le 2 A/\si^*_1 = 2A$. 
Since $\si^*$ is increasing and tends to $\infty$, by Lemma \ref{lem:log-convex}, 
there is an integer $p \ge 1$ such that
\begin{equation}\label{testfunctionconstructionequ1}
\frac{2A}{\si^*_{p+1}}\le \et \le \frac{2A}{\si^*_p}.
\end{equation}
By Lemma \ref{sequenceconstruction} and \cite[Theorem 1.3.5]{Hoermander83I} (cf.\ \cite[p.14-15]{ChaumatChollet94}), 
there exists a smooth function $\psi_{\et}$ with support contained in $[-2,2]$ satisfying 
$0 \le \ps_\et \le 1$, $\ps_\et(t) = 1$ if $t \in [-1,1]$, and
\begin{align*}
  |\ps_\et^{(k)}(t)| 
    \le 2^{k-1}\alpha_k^p &\le 
    \Big(h_{s}\Big(\frac{1}{3\si^*_{p}}\Big)\Big)^{-1}\Big(\frac{2A}{\si^*_{p+1}}\Big)^k \dot N_k
  \le \frac{\eta^k \dot N_k}{h_s(\et/(6A))}, 
\end{align*}
by \eqref{testfunctionconstructionequ1}.
For $\eta>2A$ we put $\ps_{\eta}:=\psi_{2A}$; then since $h_s \le1$,
\begin{align*}
|\ps^{(k)}_{\eta}(t)| &\le \frac{(2A)^k \dot N_k}{h_s(2A/(6A))} 
  \le \frac{1}{h_s(1/3))}  \frac{\eta^k \dot N_k}{h_s(\et/(6A))}.
\end{align*}
If $\de:= 1/h_s(1/3)$ then for every $\eta>0$,
\begin{equation*}
|\psi^{(k)}_{\eta}(t)| \le \frac{(\de \eta)^k \dot N_k}{h_s(\et/(6A))}.
\end{equation*}
The statement follows with $B:= 1/(6\de A)$ if we set $\ep = \de \et$ and $\ph_{\ep,2}:= \ps_{\ep/\de}$.
\end{proof}

Before we continue the construction of the partition of unity 
let us specify suitable weight matrices.

\begin{definition}[Admissible weight matrix] \label{admissiblematrix}
  A weight matrix $\fN$ is called \emph{admissible} if the following conditions hold.
  \begin{enumerate}
    \item For all $N, \dot N \in \fN$ we have $\nu \lesssim \dot \nu$ or $\dot \nu \lesssim \nu$. 
    \item $\sum_k 1/\nu_k < \infty$ for each $N \in \fN$.  
    \item \eqref{ugly} holds for each $N \in \fN$.
    \item For each $N \in \fN$ there is $\dot N \in \fN$ such that $\nu_k \lesssim \dot N_k^{1/k}$. 
    \item For each $N \in \fN$ there is $\dot N \in \fN$ such that $\nu_{2k} \lesssim \dot \nu_k$.
  \end{enumerate} 
\end{definition}

We remark that (4) and (5) imply that for each $N \in \fN$ there is $\ddot N \in \fN$ such that 
$\nu_k \lesssim \ddot N_k^{1/k}$ \emph{and} $\nu_{2k} \lesssim \ddot \nu_k$ which we shall frequently use; 
indeed, 
$\nu_k \le \nu_{2k} \lesssim \dot \nu_k \lesssim \ddot N_k^{1/k} \le \ddot \nu_k$, by \eqref{mucompare}. 

\begin{remark}
	The relatively strong condition (3) is needed for technical reasons 
	(in Lemma \ref{sequenceconstruction} and Proposition \ref{testfunctionconstruction}). 
	There are situations in which the extension result holds although (3) is violated; see the end of Section \ref{example}. 
\end{remark}

\begin{example} \label{single}
  A weight matrix $\fN=\{N\}$ which consists of a single weight sequence $N$ is admissible if and only if $N$ 
  is non-quasianalytic and has moderate growth. This follows from Lemma \ref{ChaumatCholletmod}.
\end{example}

\begin{definition}[Admissible weight function]
  We say that a weight function $\om$ is \emph{admissible} if and only if the associated  
  weight matrix $\fW$ is admissible. That means that $\om$ is 
  non-quasianalytic (i.e.\ $\int_1^\infty t^{-2}\om(t) \,dt < \infty$) and $\fW$ satisfies \ref{admissiblematrix}(3)\&(4);
  the other conditions in Definition \ref{admissiblematrix} hold automatically, 
  see Lemma \ref{lemma4} and \cite[Corollary 5.8]{RainerSchindl12}.
  Note that a non-quasianalytic weight function $\om$ is admissible if it satisfies \eqref{om6} (see 
  Lemma \ref{ChaumatCholletmod} and \cite[Lemma 5.7]{RainerSchindl12}).
\end{definition}

\begin{lemma} \label{iteration}
  Let $\fN$ be an admissible weight matrix, and let $n_0 \in \N$. 
  For every $N \in \fN$ there exists $\dot N \in \fN$ such that the descendants $S$ of $N$ and $\dot S$ of $\dot N$ 
  satisfy
  \begin{equation} \label{test2}
       h_s(t) \le h_{\dot s}(A t)^{n_0}, \quad t>0,
  \end{equation}
  for some constant $A = A(n_0,N)$. 
\end{lemma}

\begin{proof}
  Let $N\in \fN$ be fixed. There exist $\dot N, \ddot N, \ldots \in \fN$ such that 
  \begin{equation*}
    \nu_{2k} \lesssim \dot \nu_k \le \dot \nu_{2k} \lesssim \ddot \nu_k \le \ddot \nu_{2k} \lesssim \cdots   
  \end{equation*}
  and the same relations hold for the respective descendants, by Lemma \ref{lem:log-convex}, and so 
  \begin{equation*}
    \si^*_{2k} \lesssim \dot \si^*_k \le \dot \si^*_{2k} \lesssim \ddot \si^*_k 
    \le \ddot \si^*_{2k} \lesssim \cdots   
  \end{equation*}
  By Lemma \ref{claim1}, there are constants $A,\dot A, \ldots \ge 1$ such that 
  \begin{align*}
    h_s(t) \le h_{\dot s}(A t)^2 \le h_{\ddot s}(A \dot A t)^4 \le h_{\dddot s}(A \dot A \ddot A t)^8 \le \cdots   
  \end{align*}
  After finitely many iterations we obtain \eqref{test2}.
\end{proof}

Now we are ready to finish the construction. We will use the following lemma.
We denote by $B(x,r) = \{y \in \R^n : |x-y| <r\}$ the open ball centered at $x \in \R^n$ with radius $r>0$ 
and 
by $d(x,E) = \inf\{|x-y| : y \in E\}$ the Euclidean distance of $x$ to some set $E \subseteq \R^n$. 

\begin{lemma}[{\cite[Proposition 5]{ChaumatChollet94},\cite{CoifmanWeiss71}}] \label{balls}
  There exist constants $0<a<1$, $b>1$, $c>1$, $n_0\in \N_{>1}$, 
  such that for all compact $E \subseteq \R^n$ there is a family of open balls 
  $\{B(x_i,r_i)\}_{i \in \N}$ 
  with the following properties:
  \begin{enumerate} 
    \item $\R^n \setminus E = \bigcup_i B(x_i,r_i) = \bigcup_i B(x_i,cr_i)$,
    \item for $x \in B(x_i,c r_i)$ we have $a r_i \le d(x,E) \le b r_i$ and $a\, d(x,E) \le d(x_i,E) \le b\, d(x,E)$,
    \item for each $j$ the ball $B(x_j,cr_j)$ intersects at most $n_0$ balls of the collection $\{B(x_i,cr_i)\}_{i \in \N}$.
  \end{enumerate} 
\end{lemma}

\begin{proposition}[Partition of unity]\label{Proposition6matrix}
Let $E \subseteq \R^n$ be a compact set and let $\{B(x_i,r_i)\}_{i \in \N}$ be the family of balls 
provided by Lemma \ref{balls}.
Let $\fN$ be an admissible weight matrix, $N \in \fN$, and 
$S$ the descendant of $N$.
Then there exists $\ddot N \in \fN$ and  
$B_1\ge 1$ such that for all $\ep>0$ there is a family of $C^\infty$-functions 
$\{\vh_{i,\ep}\}_{i \in \N}$ satisfying
\begin{enumerate}
\item $0\le\vh_{i,\ep}\le 1$ for all $i\in\N$,
\item $\supp\vh_{i,\ep}\subseteq B(x_i,cr_i)$ for all $i\in\N$,
\item $\sum_{i\in\N}\vh_{i,\ep}(x)=1$ for all $x \in \R^n\setminus E$,
\item for all $\be \in\N^n$ and $x\in\R^n\setminus E$,
\begin{equation*}
  |\vh^{(\be)}_{i,\ep}(x)| \le \frac{\ep^{|\be|} \ddot N_{|\be|}}{h_s(B_1\ep\, d(x,E))}.  
\end{equation*}
\end{enumerate}
\end{proposition}

\begin{proof} 
  By Lemma \ref{iteration}, 
  there is $\dot N \in \fN$ such that its descendant $\dot S$ satisfies \eqref{test2}. 
  There is $\ddot N \in \fN$ such that $\dot \nu_k \lesssim \ddot N_k^{1/k}$
  and $\dot \nu_{2k} \lesssim \ddot \nu_k$. 
  Let $\ph_{\ep,t}$ be the functions from Proposition \ref{testfunctionconstruction} applied to $\dot \nu$ and $\ddot \nu$, 
  in particular, 
  \begin{equation*} 
    |\ph^{(\be)}_{\ep,t}(x)|\le \frac{\ep^{|\be|} \ddot N_{|\be|}}{h_{\dot s}(B\ep(t-1))}.
  \end{equation*}
   Set 
   \[
      \ps_{i,\ep}(x):=\ph_{\ep r_i/n_0,c}\Big(\frac{x-x_i}{r_i}\Big),
   \]
   where $n_0$ and $c$ are the constants from Lemma \ref{balls}, and define 
  \[
      \vh_{1,\ep}:=\psi_{1,\ep}, \quad \vh_{j,\ep}:=\ps_{j,\ep} \prod_{k=1}^{j-1}(1-\psi_{k,\ep}), ~ j\ge 2.
  \] 
  It is easy to check that (1)--(3) are satisfied (cf.\ \cite{ChaumatChollet94} for details). 
  To see (4) observe that, by Lemma \ref{balls},  
  \begin{align*}
    |\ps^{(\be)}_{i,\ep}(x)| \le \frac{(\ep/n_0)^{|\be|} \ddot N_{|\be|}}{h_{\dot s}(B\ep r_i (c-1)/n_0)} 
    \le \frac{(\ep/n_0)^{|\be|} \ddot N_{|\be|}}{h_{\dot s}(B\ep (c-1) (n_0 b)^{-1} d(x,E))}.
  \end{align*}
  Since in the product defining $\vh_{j,\ep}$ at most $n_0$ factors are different from $1$, we get
  \begin{equation*}
    |\vh^{(\be)}_{i,\ep}(x)|\le \frac{\ep^{|\be|} \ddot N_{|\be|}}{h_{\dot s}(B\ep (c-1) (n_0 b)^{-1} d(x,E))^{n_0}}. 
  \end{equation*}
  By \eqref{test2}, we obtain (4) with $B_1 = B(c-1)/(A n_0 b)$.
\end{proof}

\section{The extension theorem}

\subsection{Preliminaries}

Let $E \subseteq \R^n$ be a compact set.
Let $S=(S_k)$ be a weight sequence such that $\si_k^* = \si_k/k$ is increasing and let
$F= (F^\al)_{\al}$ be a Whitney jet of class $\cB^{\{S\}}$ on $E$, i.e.,  
there exist $C>0$ and $\rh \ge 1$ 
such that 
\begin{gather}
  |F^\al(a)| \le C \rh^{|\al|} \,  S_{|\al|}, \quad \al \in \N^n,~ a \in E,
   \label{jets1}
  \\
  |(R^p_a F)^\al(b)| \le C \rh^{p+1} \, |\al|!\, s_{p+1}\,  |b-a|^{p+1-|\al|}, \quad p \in \N,\, |\al| \le p,~ a,b \in E,  
  \label{jets2}
\end{gather}
where $s_k = \si^*_1 \cdots \si^*_k$. The next lemma is straightforward; for details see 
\cite[Proposition 10]{ChaumatChollet94}.

\begin{lemma} \label{proposition10}
  For $a_1,a_2 \in E$, $x \in \R^n$ and $|\al| \le p$,
  \begin{equation}\label{prop101}
    |(T^p_{a_1}F-T^p_{a_2}F)^{(\al)}(x)|\le C (2n^2 \rh)^{p+1}| \al|! \, s_{p+1}(|a_1-x|+|a_1-a_2|)^{p+1-|\al|}.
  \end{equation}
\end{lemma}

For every $x \in \R^n$ we denote by $\hat x$ some point in $E$ with $|x-\hat x| = d(x,E)$.
For simplicity of notation we shall use the abbreviation $d(x) := d(x,E)$. 
We need a variant of \cite[Proposition 9]{ChaumatChollet94}.

\begin{lemma}\label{proposition9}
Let $S$ and $\dot S$ be weight sequences such that $\si_k^*$ and 
$\dot \si_k^*$ are increasing and
satisfying $\si_{2k} \lesssim \dot \si_k$.
There is a constant $D_1 = D_1(S,\dot S) >1$ such that,
for all Whitney jets $F= (F^\al)_{\al}$ of class $\cB^{\{S\}}$ that satisfy \eqref{jets1} and \eqref{jets2},
all $L \ge D_1 \rh$, all $x \in \R^n$, and $\al\in \N^n$, 
\begin{align}
  |(T_{\hat x}^{2 \Ga_{\dot s} (L d(x))} F)^{(\al)}(x)| &\le C (2L)^{|\al|+1} S_{|\al|},  \label{prop91} 
  \intertext{and, if $|\al| < 2 \Ga_{\dot s}(L d(x))$,}
  |(T_{\hat x}^{2 \Ga_{\dot s}(L d(x))}F)^{(\al)}(x)-F^\al(\hat x)| &\le C (2L)^{|\al|+1} |\al|!\, s_{|\al|+1} d(x).
  \label{prop92}
\end{align}
\end{lemma}

\begin{proof}
For \eqref{prop91} we may restrict to the case $|\al| \le 2 \Ga_{\dot s}(L d(x))$. By \eqref{jets1},
\begin{align}
  |(T_{\hat x}^{2\Ga_{\dot s}(Ld(x))}F)^{(\al)}(x)|
  &\le\sum_{\substack{\al \le \be\\ |\be|\le2\Ga_{\dot s}(Ld(x))}} 
  \frac{|x-\hat{x}|^{|\be|-|\al|}}{(\be-\al)!} C \rh^{|\be|} S_{|\be|}
  \notag \\
  &\le C |\al|! \sum_{\substack{\al \le \be\\ |\be|\le2\Ga_{\dot s}(Ld(x))}}
  \frac{|\be|!\,(n d(x))^{|\be|-|\al|}}{|\al|!\, (|\be|-|\al|)!}  \rh^{|\be|} s_{|\be|}
  \notag \\
  &\le  \frac{C |\al|! }{(n d(x))^{|\al|}} \sum_{\substack{\al \le \be\\ |\be|\le2\Ga_{\dot s}(Ld(x))}}
  (2n \rh  d(x))^{|\be|} s_{|\be|}
  \notag \\
  &\le \frac{C|\al|! }{(n d(x))^{|\al|}} \sum_{j = |\al|}^{2\Ga_{\dot s}(Ld(x))}
  (2n^2 \rh  d(x))^{j} s_{j},  \label{calculation} 
\end{align}
since the number of $\be \in \N^n$ with $|\be|=j$ is bounded by $n^j$.

The assumption $\si_{2k} \lesssim \dot \si_k$ is equivalent to $\si^*_{2k} \lesssim \dot \si^*_k$.
So, by Lemma \ref{claim1}, 
there is some $\la<1$ such that $2 \Ga_{\dot s}(t) \le \Ga_s(\la t)$ for all $t>0$, and thus
\begin{align*}
  |(T_{\hat x}^{2\Ga_{\dot s}(Ld(x))}F)^{(\al)}(x)|
  &\le \frac{C|\al|! }{(n d(x))^{|\al|}} \sum_{j = |\al|}^{\Ga_s(L\la d(x))}
  (2n^2 \rh  d(x))^{j} s_{j}.  
\end{align*}
By \eqref{GaProp1}, $(L \la d(x))^{j} s_{j}\le (L \la d(x))^{|\al|} s_{|\al|}$ for $|\al|\le j \le \Ga_s(L \la d(x))$,
and hence
\begin{align*}
  |(T_{\hat x}^{2 \Ga_{\dot s}(Ld(x))}F)^{(\al)}(x)|
  &\le C S_{|\al|} \Big(\frac{L \la}{n}\Big)^{|\al|} \sum_{j = |\al|}^{\Ga_s(L\la d(x))}
  \Big(\frac{2n^2 \rh}{L \la } \Big)^{j}.  
\end{align*}
We obtain \eqref{prop91} if $L$ is chosen such that $2n^2\rh /(L\la)\le 1/2$; 
then $D_1= 4n^2/\la$.

For \eqref{prop92} note that, if $|\al| < 2\Ga_{\dot s}(L d(x))$,
\[
  (T_{\hat x}^{2 \Ga_{\dot s}(L d(x))}F)^{(\al)}(x)-F^\al(\hat x) 
  = \sum_{\substack{\al \le \be\\ |\al| < |\be|\le2 \Ga_{\dot s}(Ld(x))}}\frac{(x-\hat{x})^{\be-\al}}{(\be-\al)!} F^\be(\hat x).
\]
Thus the same arguments yield \eqref{prop92}.
\end{proof}

\subsection{The extension theorem}

\begin{theorem}[Extension theorem]\label{extensiontheorem}
Let $\fN$ be an admissible weight matrix, $N \in \fN$, and $S$ the descendant of $N$.
Let $E$ be a compact subset of $\R^n$. 
Then the jet mapping $j^\infty_E : \cB^{\{\fN\}}(\R^n) \to \cB^{\{S\}}(E)$ is surjective.
\end{theorem}

\begin{proof} 
Let $\ep,L > 0$ be given.
Since $\fN$ is admissible, there exist
\begin{itemize}
  \item $\dot N \in \fN$ such that $\nu_k \lesssim \dot N_k^{1/k}$ and $\nu_{2k} \lesssim \dot \nu_k$,
  \item $\ddot N \in \fN$ such that $\dot \nu_k \lesssim \ddot N_k^{1/k}$ and $\dot \nu_{2k} \lesssim \ddot \nu_k$,
\end{itemize}
and, by 
Lemma \ref{lem:log-convex}, we have $\si^*_{2k+1} \lesssim \si^*_{2k} \lesssim \dot \si^*_k$ and $\dot \si^*_{2k} \lesssim 
\ddot \si^*_k$, 
for the respective descendants. 
Then, by Lemma \ref{claim1} and Remark \ref{rem:claim1}, 
there are constants $B, D\ge 1$ and $\la <1$ such that 
\begin{gather} \label{mtildem}
    s_{2k+1} \le B^{2k+1} \dot s_k^2, \quad  
    \text{ for all } k, \\
    2 \Ga_{\dot s}(t) \le \Ga_s (\la t), \quad \text{ for } t>0, \label{GatildeGa} \\
    h_{\dot s}(t) \le  h_{\ddot s}(Dt)^2, \quad \text{ for } t>0. \label{htildeh}
\end{gather}
Note that $\dot N, \ddot N,  S, \dot S, \ddot S$ and thus also the constants $B,D$ and $\la$ 
only depend on $N$.

Let $\{B(x_i,r_i)\}_{i\in \N}$ be the family of balls provided by Lemma \ref{balls}.
By Proposition \ref{Proposition6matrix}, there is
\begin{itemize}
  \item $\dddot N \in \fN$ and  a collection of $C^\infty$-functions $\{\vh_{i,\ep}\}_{i\in \N}$ 
satisfying \ref{Proposition6matrix}(1)--(3) and
\begin{equation} \label{estvh}
  |\vh^{(\be)}_{i,\ep}(x)| 
  \le \frac{\ep^{|\be|} \dddot N_{|\be|}}{h_{\ddot s}(B_1\ep\, d(x))}, \quad \be \in \N^n,~ x \in \R^n\setminus E,
\end{equation}
for some constant $B_1=B_1(N)$.
\end{itemize}

Let $F = (F^\al)_\al$ be a Whitney jet of class $\cB^{\{S\}}$ on $E$ satisfying \eqref{jets1} and \eqref{jets2}.
We define 
\[
  f(x) := 
  \begin{cases}
    \sum_{i\in \N} \vh_{i,\ep}(x) \, T_{\hat x_i}^{2 \Ga_{\dot s}(L d(x_i)) } F(x),  & \text{ if } x \in \R^n \setminus E, \\
    F^0(x), & \text{ if } x \in  E.
  \end{cases}
\]
Clearly, $f$ is $C^\infty$ in $\R^n \setminus E$.
The theorem will follow from the following claim.

\begin{claim*} 
  There are constants $K_i= K_i(N)$, $i=1,\ldots,4$, such that the following holds.
  If $\ep = K_1 L$ and $L>K_2 \rh$, then for all $x \in \R^n \setminus E$ with $d(x) < 1$ and all $\al \in \N^n$,
  \begin{equation} \label{eqclaim2}
    |\p^\al (f - T_{\hat x}^{2 \Ga_{\dot s}(L d(x)) } F) (x)| \le C (L K_3)^{|\al|+1} \dddot N_{|\al|} h_{\ddot s}(LK_4 d(x)); 
  \end{equation}
  $C$ and $\rh$ are the constants from \eqref{jets1} and \eqref{jets2}.
\end{claim*}

In fact, let us assume that the claim holds.
We may additionally assume that $L \ge D_1 \rh$ for the constant $D_1$ in Lemma \ref{proposition9}. So, by \eqref{prop91} and 
 \eqref{eqclaim2}, for $x \in \R^n \setminus E$ with $d(x)<1$ and $\al \in \N^n$, 
 \begin{align}
    |f^{(\al)}(x)| 
    &\le 
   |(T_{\hat x}^{2 \Ga_{\dot s}(L d(x))} F)^{(\al)}(x)| + |\p^\al (f - T_{\hat x}^{2 \Ga_{\dot s}(L d(x)) } F) (x)|
   \notag \\ \label{final}
   &\le 
    C (L K)^{|\al|+1} \dddot N_{|\al|}
 \end{align}
 for a suitable constant $K=K(n,N)$, because $h_{\ddot s} \le 1$ and $\si \lesssim \dddot \nu$.

 Let us fix a point $a \in E$ and $\al\in \N^n$. 
 Since $\Ga_{\dot s}(t) \to \infty$ as $t \to 0$, we have $|\al| < 2 \Ga_{\dot s}(L d(x))$ if $x \in \R^n \setminus E$ is 
 sufficiently close to $a$. Thus, as $x \to a$, 
 \begin{align*}
   &|f^{(\al)}(x) - F^{\al}(a)|
   \\
   &\le 
   |\p^\al (f - T_{\hat x}^{2 \Ga_{\dot s}(L d(x)) } F) (x)| + 
   |(T_{\hat x}^{2 \Ga_{\dot s}(L d(x))}F)^{(\al)}(x)-F^\al(\hat x)| + |F^\al(\hat x) - F^\al(a)| 
   \\
   & = O(h_{\ddot s}(L K_4 d(x))) + O(d(x)) + O(|\hat x - a|),
 \end{align*}
 by \eqref{jets2}, \eqref{prop92}, and \eqref{eqclaim2}. 
 Hence $f^{(\al)}(x) \to  F^{\al}(a)$ as $x \to a$.
 We may conclude that $f \in C^\infty(\R^n)$.
 After multiplication with a suitable cut-off function of class $\cB^{\{\dddot N\}}$ with support in $\{x : d(x) < 1\}$,
 we find that $f \in \cB^{\{\dddot N\}}(\R^n)$ thanks to \eqref{jets1} and \eqref{final}.
 The result follows.

\bigskip

\paragraph{\emph{Proof of the claim}}

By the Leibniz rule,
\begin{align}
  \p^\al& (f - T_{\hat x}^{2 \Ga_{\dot s}(L d(x)) } F) (x) \notag
  \\
  &=
  \sum_{\be \le \al} \binom{\al}{\be} \sum_i \vh_{i,\ep}^{(\al-\be)}(x) \, 
  \p^\be (T_{\hat x_i}^{2 \Ga_{\dot s}(L d(x_i)) } F - T_{\hat x}^{2 \Ga_{\dot s}(L d(x)) } F) (x). \label{Leibniz}
\end{align}
Let us estimate $\p^\be (T_{\hat x_i}^{2 \Ga_{\dot s}(L d(x_i)) } F - T_{\hat x}^{2 \Ga_{\dot s}(L d(x)) } F) (x) = H_1 + H_2$ 
for $x \in B(x_i,c r_i)$, where 
\begin{align*}
  H_1 &:= \p^\be (T_{\hat x_i}^{2 \Ga_{\dot s}(L d(x_i)) } F - T_{\hat x}^{2 \Ga_{\dot s}(L d(x_i)) } F) (x),
  \\
  H_2 &:= \p^\be (T_{\hat x}^{2 \Ga_{\dot s}(L d(x_i)) } F - T_{\hat x}^{2 \Ga_{\dot s}(L d(x)) } F) (x).
\end{align*}

\medskip

\paragraph{\emph{Estimation of $H_1$}} It suffices to consider $|\be| \le 2 \Ga_{\dot s}(L d(x_i)) =: 2p$. 
By Lemma \ref{proposition10}, 
\begin{align*}
  |H_1| 
  \le C (2n^2 \rh)^{2 p+1} |\be|! \, 
  s_{2 p+1}  (|\hat x_i-x|+|\hat x_i-\hat x|)^{2 p+1-|\be|}.
\end{align*}
By Lemma \ref{balls}(2), for $x \in B(x_i,c r_i)$,
\begin{gather*}
  |\hat x_i-x|\le|\hat x_i-x_i|+|x_i-x|\le d(x_i)+ cr_i \le (1+c/a)d(x_i), \\
  |\hat x_i-\hat x|\le|\hat x_i-x|+|x-\hat x|\le(1+(c+1)/a)d(x_i).
\end{gather*}
If we set $K := 2 (1+(c+1)/a)$ and use \eqref{mtildem}, we obtain 
\begin{align*}
  |H_1| 
  &\le C (2n^2 B \rh)^{2 p+1} 
  |\be|! \, \dot s_{p}^2(K d(x_i))^{2 p+1-|\be|}.
\end{align*} 
Since $h_{\dot s}(L d(x_i)) = \dot s_{p} (L d(x_i))^{p} \le \dot s_{|\be|} (L d(x_i))^{|\be|}$,
by \eqref{GaProp2}, 
and $d(x_i) \le b\, d(x)$, by Lemma \ref{balls}(2),
\begin{align*}
  |H_1| 
  &\le C 2n^2 B K \rh \Big(\frac{2n^2 B K \rh}{L}\Big)^{2 p}  \, 
  b\, d(x)\, |\be|!\,\dot s_{|\be|} L^{|\be|}\, h_{\dot s}(L d(x_i)).
\end{align*}
If $L >  2n^2 B Kb \, \rh$ and $d(x)<1$, then 
\begin{align} \label{H1}
  |H_1| 
  &\le C    
    L^{|\be|+1}  \dot S_{|\be|} \, h_{\dot s}(L d(x_i)).
\end{align}

\medskip

\paragraph{\emph{Estimation of $H_2$}}
Here we differentiate a polynomial $T_{\hat x}^{2 \Ga_{\dot s}(L d(x_i)) } F - T_{\hat x}^{2 \Ga_{\dot s}(L d(x)) } F$
of degree at most $2 \Ga_{\dot s}(La d(x)) \le \Ga_s(L \la a d(x))$, by Lemma \ref{balls}(2) (as $\Ga_{\dot s}$ is decreasing)
and \eqref{GatildeGa}. Again by Lemma \ref{balls}(2), the valuation of the polynomial is at least 
$2 \Ga_{\dot s}(Lb d(x)) =: 2 q$.
Thus, by the calculation in \eqref{calculation}, 
\begin{align*}
  |H_2| 
  &\le \frac{C |\be|!}{(n d(x))^{|\be|}} \sum_{j =2 q }^{\Ga_s(L \la a d(x))}
  (2n^2 \rh  d(x))^{j} s_{j}. 
\end{align*}
By \eqref{GaProp1}, 
$s_j (L \la a d(x))^j \le s_{2 q} (L \la a d(x))^{2 q}$,   
for $j$ in the above sum, and by \eqref{GaProp2}, 
$h_{\dot s}(Lb d(x)) = \dot s_{q} (Lb d(x))^{q} \le \dot s_{|\be|} (Lb d(x))^{|\be|}$.
Hence, using \eqref{mtildem}, we find
\begin{align*}
  |H_2| 
  &\le \frac{C |\be|!}{(n d(x))^{|\be|}} \sum_{j =2 q }^{\Ga_s(L \la a d(x))}
  \Big(\frac{2n^2 \rh}{L\la a} \Big)^{j} s_{2 q} (L \la a d(x))^{2 q}
  \\
  &\le \frac{CB |\be|!}{(n d(x))^{|\be|}} \sum_{j =2 q }^{\Ga_s(L \la a d(x))}
  \Big(\frac{2n^2 \rh}{L\la a} \Big)^{j} \dot s_{q}^2 (B L \la a d(x))^{2 q}
  \\
  &\le  CB \Big(\frac{L b}{n}\Big)^{|\be|} |\be|!\, \dot s_{|\be|}  h_{\dot s}(Lb d(x)) 
  \Big(\frac{\la a}{b}\Big)^{2 q}
  \sum_{j =2 q }^{\Ga_s(L \la a d(x))}
  \Big(\frac{2n^2 B \rh}{L\la a} \Big)^{j}.   
\end{align*}
If we choose $L \ge \frac{4n^2 B \rh}{\la a}$ then the sum is bounded by $2$. 
Let us furthermore assume that $L > 2nB/b$. Then, as $\la<1$, $a<1$, $b>1$,
\begin{align} \label{H2}
  |H_2| 
  &\le  C \Big(\frac{L b}{n}\Big)^{|\be|+1} \dot S_{|\be|}  h_{\dot s}(Lb d(x)).   
\end{align}  

\medskip

Let us finish the proof of the claim. By \eqref{H1} and \eqref{H2}, for $x \in B(x_i,c r_i)$ with $d(x)<1$,
using Lemma \ref{balls}(2) and the fact that $h_{\dot s}$ is increasing,
\begin{align*}
  |\p^\be (T_{\hat x_i}^{2 \Ga_{\dot s}(L d(x_i)) } F - T_{\hat x}^{2 \Ga_{\dot s}(L d(x)) } F) (x)|
  \le C    
    (2bL)^{|\be|+1}  \dot S_{|\be|} \, h_{\dot s }(Lb d(x)).
\end{align*}
Thus, by \eqref{estvh}, \eqref{Leibniz}, and Lemma \ref{balls}(3), 
\begin{align*}
  |\p^\al& (f - T_{\hat x}^{2 \Ga_{\dot s}(L d(x)) } F) (x)| 
  \\
  &\le
  \sum_{\be \le \al} \frac{\al!}{\be!(\al-\be)!} 
  \cdot n_0 \cdot 
   \frac{\ep^{|\al|-|\be|} \dddot N_{|\al|-|\be|}}{h_{\ddot s}(B_1\ep\, d(x))}
  \cdot C    (2bL)^{|\be|+1}  \dot S_{|\be|} \, h_{\dot s}(Lb d(x))
  \\
  &\le  C n_0
  \sum_{j=0}^{|\al|} \frac{|\al|!\, n^{|\al|+j}}{j!(|\al|-j)!}  
   \ep^{|\al|-j} (2bL)^{j+1} \dddot N_{|\al|-j}
        \dot S_{j} \, \frac{h_{\dot s}(Lb d(x))}{h_{\ddot s}(B_1\ep\, d(x))}     
  \\
  &\le  2bL C n_0 n^{|\al|} \dddot N_{|\al|}  \frac{h_{\dot s}(Lb d(x))}{h_{\ddot s}(B_1\ep\, d(x))}
  \sum_{j=0}^{|\al|} \frac{|\al|!\, }{j!(|\al|-j)!}  
   \ep^{|\al|-j} 
      (2bL n A)^{j}
  \\
  &=  2bL C n_0 (n (\ep + 2bL n A))^{|\al|} \dddot N_{|\al|}  \frac{h_{\dot s}(Lb d(x))}{h_{\ddot s}(B_1\ep\, d(x))},            
\end{align*}
since $\dot \si \lesssim \dot \nu  \lesssim \dddot \nu$, whence
$\dot S_j \le A^j \dddot N_j$, and since $\dddot N_{|\al|-j} \dddot N_{j} \le \dddot N_{|\al|}$.
Let us fix $L$, according to the restrictions above, and set $\ep := Lb D/B_1$, where $D$ is the constant from \eqref{htildeh}. 
Then, by \eqref{htildeh}, 
\begin{equation*}
   \frac{h_{\dot s}(Lb d(x))}{h_{\ddot s}(B_1\ep\, d(x))} 
   = \frac{h_{\dot s}(Lb d(x))}{h_{\ddot s}(DLb d(x))} 
   \le  h_{\ddot s}(DLb d(x)),
 \end{equation*} 
 and we obtain \eqref{eqclaim2}. The claim is proved.
\end{proof}

\begin{remark}
  The proof of Theorem \ref{extensiontheorem} shows that for each $\rh>0$ there is a continuous linear extension operator 
  $\cB^S_\rh(E) \to \cB^{\dddot N}_{K\rh}(\R^n)$ for a suitable constant $K$. This extension operator depends on 
  $\rh$ (through $L$ and $\ep$) and in general there is no continuous extension operator 
  $\cB^{\{S\}}(E) \to \cB^{\{\fN\}}(\R^n)$, 
  cf.\ \cite{Petzsche88} and 
  \cite[p.\,223]{schmetsValdivia00}. 
\end{remark}

\subsection{Applications}

\begin{corollary} \label{cor1:extension}
Let $\fN$ be an admissible weight matrix.
Let $\fM$ be a weight matrix 
such that for all $M \in \fM$ there is $N \in \fN$ with 
$\sum_{\ell \ge k} 1/\nu_\ell \lesssim k/ \mu_k$ and $\mu \lesssim \nu$. 
Let $E$ be a compact subset of $\R^n$. 
Then the jet mapping $j^\infty_E : \cB^{\{\fN\}}(\R^n) \to \cB^{\{\fM\}}(E)$ is surjective.
\end{corollary}

\begin{proof}
  Let $M \in \fM$ be fixed.
  Lemma \ref{lem:log-convex} implies $\mu \lesssim \si \lesssim \nu$, where $\si$ is the descendant of 
  $\nu$. By Theorem \ref{extensiontheorem}, 
  $j^\infty_E : \cB^{\{\fN\}}(\R^n) \to \cB^{\{S\}}(E)$ is surjective and $\cB^{\{M\}}(E) \subseteq \cB^{\{S\}}(E)$.
\end{proof}

\begin{corollary}[Extension preserving the class] \label{cor2:extension}
Let $\fN$ be an admissible weight matrix such that for all $N \in \fN$ there is $\dot N \in \fN$ with 
$\sum_{\ell \ge k} 1/\dot \nu_\ell \lesssim k/ \nu_k$. 
Let $E$ be a compact subset of $\R^n$. 
Then the jet mapping $j^\infty_E : \cB^{\{\fN\}}(\R^n) \to \cB^{\{\fN\}}(E)$ is surjective.
\end{corollary}

\begin{proof}
  This is a special case of Corollary \ref{cor1:extension}.
\end{proof}

If $\fN$ consists just of a single weight sequence we recover a slightly sharper version of the result of 
Chaumat and Chollet \cite[Theorem 30]{ChaumatChollet94}.

\begin{corollary}
  Let $N$ be a non-quasianalytic weight sequence of moderate growth. 
  Then the descendant $S$ of $N$ has moderate growth. 
  The mapping $j^\infty_E : \cB^{\{N\}}(\R^n) \to \cB^{\{S\}}(E)$ is surjective for every compact $E \subseteq \R^n$.
\end{corollary}

\begin{proof}
  That $S$ has moderate growth follows from Lemma \ref{ChaumatCholletmod} and Lemma \ref{lem:log-convex}(6)
  (applied to $\nu=\dot \nu$). 
\end{proof}

Chaumat and Chollet show that if $M$ is a weight sequence of moderate growth such that $\mu^*$ is increasing 
and $N$ is a non-quasianalytic weight sequence with $\mu \lesssim \nu$ then the following are equivalent:
\begin{itemize}
  \item $j^\infty_E : \cB^{\{N\}}(\R^n) \to \cB^{\{M\}}(E)$ is surjective for every compact $E \subseteq \R^n$. 
  \item $j^\infty_{\{0\}} : \cB^{\{N\}}(\R^n) \to \cB^{\{M\}}(\{0\})$ is surjective. 
  \item $\sum_{\ell \ge k} 1/\nu_\ell \lesssim k/\mu_k$.
\end{itemize}
In the situation of the corollary 
we see, by Lemma \ref{lem:log-convex}(5), that, for arbitrary $E$, $\cB^{\{S\}}(E)$ is the largest space of Whitney jets 
among the $\cB^{\{M\}}(E)$ which is contained in $j^\infty_E \cB^{\{N\}}(\R^n)$.

Let us collect the immediate consequences for classes defined by weight functions.

\begin{corollary}
  Let $\ta$ be an admissible weight function with associated weight matrix $\fT$. 
  Assume that $\om$ is a weight function with associated weight matrix $\fW$ such that 
  for all $W \in \fW$ there is $T \in \fT$ with $\sum_{\ell \ge k} T_{\ell-1}/T_\ell \lesssim k W_{k-1}/ W_k$ 
  and $W_k/W_{k-1} \lesssim T_k/T_{k-1}$. 
  Let $E$ be a compact subset of $\R^n$. 
  Then the jet mapping $j^\infty_E : \cB^{\{\ta\}}(\R^n) \to \cB^{\{\om\}}(E)$ is surjective.
\end{corollary}

\begin{corollary} \label{cor:mixed}
  Let $\om$ be an admissible weight function with associated weight matrix $\fW$ such that 
  for all $W \in \fW$ there is $\dot W \in \fW$ with $\sum_{\ell \ge k} 1/\dot \vt_\ell \lesssim k / \vt_k$. 
  Let $E$ be a compact subset of $\R^n$. 
  Then the jet mapping $j^\infty_E : \cB^{\{\om\}}(\R^n) \to \cB^{\{\om\}}(E)$ is surjective.
\end{corollary}

\subsection{Characterization of the extension property} \label{ssec:characterization}

In this section we prove a converse to Corollary \ref{cor2:extension}, 
using a result of Schmets and Valdivia \cite{SchmetsValdivia04}.

For weight sequences $M=(M_k)$ and $N=(N_k)$ and positive integers $p$ and $k$ set
\begin{equation*}
  \vh^{M,N}_{p,k} := \sup_{0 \le j < k} \Big(\frac{M_k}{p^k N_j}\Big)^{1/(k-j)}.
\end{equation*}
and consider the condition 
\begin{equation} \label{SV}
     \sum_{j \ge k} \frac{1}{\nu_j} \lesssim \frac{k}{\vh^{M,N}_{p,k}}. 
\end{equation}
Provided that $M \le N$ we have $\vh^{M,N}_{p,k} \le \mu_k$ for every positive integer $p$, indeed,
\[
    \Big(\frac{M_k}{p^k N_j}\Big)^{1/(k-j)} \le \Big(\frac{M_k}{N_j}\Big)^{1/(k-j)} \le (\mu_{j+1} \cdots \mu_k)^{1/(k-j)} 
    \le \mu_k.  
\]
Thus $\sum_{\ell \ge k}  1/\nu_\ell \lesssim k/\mu_k$ implies \eqref{SV} for every 
$p \in \N_{>0}$.  
A partial converse holds for suitable weight matrices.

\begin{lemma} \label{SVgamma1}
  Let $\fN$ be a weight matrix satisfying
  \begin{equation} \label{SVass}
    \A N \in \fN \E \dot N \in \fN : \nu_{k} \lesssim \dot N_k^{1/k}.
  \end{equation} 
  Then the following are equivalent:
  \begin{gather}
    \A N \in \fN \E \dot N \in \fN \E p \in \N_{>0} : 
      \sum_{\ell \ge k} 
     \frac{1}{\dot \nu_\ell} \lesssim \frac{k}{\vh^{N,\dot N}_{p,k}}. \label{SVQ}
     \\
     \A N \in \fN \E \dot N \in \fN :
      \sum_{\ell \ge k}  \frac{1}{\dot \nu_\ell} \lesssim \frac{k}{\nu_k}. \label{CCQ}
  \end{gather}  
\end{lemma}

\begin{proof}
  That \eqref{CCQ} implies \eqref{SVQ} is clear by the arguments above.
  Suppose that \eqref{SVQ} holds and let $N \in \fN$ be fixed. 
  Then, by \eqref{SVass} and \eqref{SVQ}, there exist 
  $\dot N, \ddot N \in \fN$ such that 
  $\nu_{k} \lesssim \dot N_k^{1/k} \le p\, \vh^{\dot N,\ddot N}_{p,k}$ and  
  $\sum_{j \ge k} 1/\ddot \nu_j \lesssim k/\vh^{\dot N,\ddot N}_{p,k}$ which entails \eqref{CCQ}. 
\end{proof}

\begin{proposition} \label{converse}
  Let $\fN$ be an admissible weight matrix. 
  The inclusion $\cB^{\{\fN\}}(\{0\}) 
  \subseteq j^\infty_{\{0\}} \cB^{\{\fN\}}(\{\R\})$ 
  implies \eqref{CCQ}.	
\end{proposition}

\begin{proof}
  By Lemma \ref{SVgamma1}, it suffices to show \eqref{SVQ}. 

  \begin{claim*}
   $\cB^{\{\fN\}}(\{0\}) \subseteq j^\infty_{\{0\}} \cB^{\{\fN\}}(\{\R\})$ if and only if \eqref{SVQ}. 
  \end{claim*}

  We use the following result of Schmets and Valdivia \cite[Theorem 1.1]{SchmetsValdivia04}: 
  \emph{Let $M$ and $N$ be weight sequences such that $M_k^{1/k} \lesssim N_k^{1/k}$ and $N$ is non-quasianalytic. 
     Then $\cB^{\{M\}}(\{0\}) \subseteq j^\infty_{\{0\}} \cB^{\{N\}}(\R)$ if and only if
     \eqref{SV} holds for some $p$.} 
  In \cite{SchmetsValdivia04} the assumptions on $M$ and $N$ are slightly more restrictive, but the same 
  proof yields the result.    

  This result implies the claim, since $\cB^{\{\fN\}}(\{0\}) \subseteq j^\infty_{\{0\}} \cB^{\{\fN\}}(\{\R\})$ entails 
  that for all $N \in \fN$ there is $\dot N \in \fN$ such that 
  $\cB^{\{N\}}(\{0\}) \subseteq j^\infty_{\{0\}} \cB^{\{\dot N\}}(\{\R\})$ 
  which follows from 
  a simple modification of the proof of \cite[Proposition 3.3]{SchmetsValdivia04}.
\end{proof}

\begin{theorem}[Characterization of the extension property] \label{characterization}
Let $\fN$ be an admissible weight matrix. 
The jet mapping $j^\infty_E : \cB^{\{\fN\}}(\R^n) \to \cB^{\{\fN\}}(E)$ is surjective 
for every compact set $E \subseteq \R^n$ if and only if 
\eqref{CCQ}.
\end{theorem}

\begin{proof}
  Corollary \ref{cor2:extension} and Proposition \ref{converse}.
\end{proof}

For weight functions this implies the following.

\begin{corollary}
	Let $\om$ be an admissible weight function. 
	Then the following are equivalent:
	\begin{enumerate}
	 	\item $j^\infty_E : \cB^{\{\om\}}(\R^n) \to \cB^{\{\om\}}(E)$ is surjective for every compact set $E \subseteq \R^n$. 
	 	\item For all $x>0$ there is $y>0$ such that $\sum_{\ell \ge k} 1/\vt^y_\ell  \lesssim k/\vt^x_k$.
	 \end{enumerate} 
\end{corollary}

We want to emphasize that \cite{BBMT91} proved the equivalence of (1) with 
\begin{enumerate}
	\item[(3)] $\int_1^\infty y^{-2}\om(ty) \, dy \le A \om(t) + B$ for positive constants $A,B$, 
\end{enumerate}
for arbitrary weight functions by different methods.

\subsection{A class of admissible weight functions} \label{example}

For $s>1$ consider the weight function (cf.\ \cite[Section 3.10]{Schindl14})
\[
  \om_s(t) := \max\{0,(\log t)^s\}.
\]
Then $\vh_s(t) = t^s$ for $t >0$ and $\vh_s(t) = 0$ for $t \le 0$. 
Let us set $r= s/(s-1)$; then $r>1$ and $r-1 = 1/(s-1)$.
The Young conjugate of $\vh_s$ is $\vh_s^* (t) = C_s \, t^{r}$ where $C_s = (s-1) s^{- r}$.
  The associated weight sequences $(W^{s,\rh}_k)_{k}$, $\rh>0$, are given by
\[
  W^{s,\rh}_k = \exp(C_s \, \rh^{r-1}\, k^{r}).
\]
   
\begin{proposition} \label{prop:omegas}
  Let $s>1$. The weight function $\om_s$ has the following properties.
  \begin{enumerate}
    \item For all $\rh > 0$ we have $\sum_{\ell \ge k} 1/\vt^{s,\rh}_{\ell} \lesssim k/\vt^{s,\rh}_k$ (thus 
    $\si^{s,\rh} \sim \vt^{s,\rh}$ if $\si^{s,\rh}$ denotes the descendant of $\vt^{s,\rh}$). 
    \item Condition \eqref{ugly} holds for $W^{s,\rh}$ if $s \ge 2$ (condition \eqref{om6} holds for no $s>1$),
    \item For all $\rh>0$ and $k \in \N_{>0}$ we have
        $\vt^{s,\rh}_{k+1} \le (W^{s,6\rh}_k)^{1/k}$.
    \item $\int_1^\infty y^{-2} \om_s(ty) \, dy \le A \om_s(t) + B$ for positive constants $A,B$,
    
  \end{enumerate}
  In particular, $\om_s$ is admissible if $s\ge 2$. 
\end{proposition}

\begin{proof}   
  The function $f(x) = x^r$ is increasing on $(0,\infty)$ with increasing 
    derivative $f'(x) = r x^{r-1}$. Thus $f'(k) \le f(k+1)-f(k) \le f'(k+1)$, i.e.,
    \begin{equation} \label{growthf}
      r k^{r-1} \le (k+1)^r-k^r \le r(k+1)^{r-1}.
    \end{equation}

  (1) By \eqref{growthf}, 
    \begin{align*}
      \frac{\vt^{s,\rh}_{2k}}{\vt^{s,\rh}_{k}} = \exp\big(C_s \rh^{r-1} (2^r-1) (k^{r} - (k-1)^r)\big) \to \infty
      \quad \text{ as $k \to \infty$,}
    \end{align*}
    which implies (1) by \cite[Proposition 1.1]{Petzsche88}. 

  (2) By (1), \eqref{ugly} for $W^{s,\rh}$ is equivalent to $\vt^{s,\rh}_{k+1} \lesssim \vt^{s,\rh}_{k}$.
  We have $s \ge 2$ if and only if $1 < r \le 2$. 
  Then the function $f'$ is concave on $(0,\infty)$ since $f'''(x) = r(r-1)(r-2) x^{r-3} \le 0$. 
  Thus (by a look at its derivative) the function $(x+1)^r + (x-1)^r - 2 x^r$ is decreasing, which implies 
  $\vt^{s,\rh}_{k+1} \lesssim \vt^{s,\rh}_{k}$. 

  (3) By \eqref{growthf},
  \begin{align*}
    \vt^{s,\rh}_{k+1} &=\exp\big(C_s \rh^{r-1} ((k+1)^{r} - k^{r})\big) 
     \le \exp\big(C_s r \rh^{r-1} (k+1)^{r-1}\big)
    \\
    & \le \exp\big(C_s (2 e \rh)^{r-1} k^{r-1}\big) \le (W^{s,6\rh}_k)^{1/k}.
    \end{align*}

    (4) This follows from (1) in view of \cite[Proposition 4.4]{Komatsu73} and \cite[Lemma 5.7]{RainerSchindl12}.
    Alternatively, it can easily be seen by checking some equivalent condition from \cite[Theorem 1.7]{BBMT91}, or directly 
    by using the asymptotic behavior of the incomplete Gamma function.
\end{proof}

Since each $\om_s$ violates \eqref{om6}, and thus the corresponding class cannot be described by a single weight sequence, 
the extension property does not follow from the result of Chaumat and Chollet. 
  Our results imply that
  the jet mapping  
  $j^\infty_E : \cB^{\{\om_s\}}(\R^n) \to \cB^{\{\om_s\}}(E)$ is surjective for every compact subset $E \subseteq \R^n$ 
  provided that $s\ge 2$. However, by \cite{BBMT91} it is so also for $1< s< 2$.


\def\cprime{$'$}
\providecommand{\bysame}{\leavevmode\hbox to3em{\hrulefill}\thinspace}
\providecommand{\MR}{\relax\ifhmode\unskip\space\fi MR }
\providecommand{\MRhref}[2]{%
  \href{http://www.ams.org/mathscinet-getitem?mr=#1}{#2}
}
\providecommand{\href}[2]{#2}

\end{document}